\documentclass[12pt,reqno]{amsart}
\usepackage[left=3cm,top=2cm,right=3cm,bottom=2cm]{geometry}
\usepackage{amssymb}
\usepackage{amsbsy}
\usepackage{epsfig}
\usepackage{tikz}

\begin{document}
\title{A Linear algorithm for obtaining the Laplacian eigenvalues of a cograph}
\author{Guantao Chen}
\address{Departament of Mathematics and Statistics, Georgia State University, Atlanta 30303 USA}
\email{\tt gchen@gsu.edu}
\author{  Fernando C. Tura }
\address{ Departamento de Matem\'atica, UFSM,97105--900 Santa Maria, RS, Brazil}
\email{\tt fernando.tura@ufsm.br}
\pdfpagewidth 8.5 in \pdfpageheight 11 in

\newcommand{\bfx}{{\mathbf x}}
\newcommand{\casei}{{\bf case~1}}
\newcommand{\subia}{{\bf subcase~1a}}
\newcommand{\subib}{{\bf subcase~1b}}
\newcommand{\subic}{{\bf subcase~1c}}
\newcommand{\caseii}{{\bf case~2}}
\newcommand{\subiia}{{\bf subcase~2a}}
\newcommand{\subiib}{{\bf subcase~2b}}
\newcommand{\caseiii}{{\bf case~3}}
\newcommand{\myvar}{x}
\newcommand{\exvar}{\frac{\sqrt{3} + 1}{2}}
\newcommand{\Prf}{{\noindent \bf Proof: }}
\newcommand{\PrfSketch}{{\bf Proof (Sketch): }}
\newcommand{\boldQ}{\mbox{\bf Q}}
\newcommand{\boldR}{\mbox{\bf R}}
\newcommand{\boldZ}{\mbox{\bf Z}}
\newcommand{\boldc}{\mbox{\bf c}}
\newcommand{\sign}{\mbox{sign}}
\newcommand{\alphaseq}{{\pmb \alpha}_{G,\myvar}}
\newcommand{\alphaseqGprime}{{\pmb \alpha}_{G^\prime,\myvar}}
\newcommand{\alphaseqlam}{{\pmb \alpha}_{G,-\lambdamin}}
\newtheorem{Thr}{Theorem}[section]
\newtheorem{Pro}{Proposition}
\newtheorem{Que}{Question}
\newtheorem{Con}{Conjecture}
\newtheorem{Cor}{Corollary}
\newtheorem{Lem}{Lemma}[section]
\newtheorem{Fac}{Fact}
\newtheorem{Ex}{Example}[section]
\newtheorem{Def}{Definition}[section]
\newtheorem{Prop}{Proposition}[section]
\def\floor#1{\left\lfloor{#1}\right\rfloor}

\newenvironment{my_enumerate}{
\begin{enumerate}
  \setlength{\baselineskip}{14pt}
  \setlength{\parskip}{0pt}
  \setlength{\parsep}{0pt}}{\end{enumerate}
}
\newenvironment{my_description}{
\begin{description}
  \setlength{\baselineskip}{14pt}
  \setlength{\parskip}{0pt}
  \setlength{\parsep}{0pt}}{\end{description}
}

\maketitle
\begin{abstract}  
In this paper, we give an $O(n)$ time and space algorithm  for obtaining the Laplacian eigenvalues of a cograph. This approach is more efficient as there is no need to directly compute the eigenvalues of Laplacian matrix related to this class of graph. As an application, we use this algorithm as a tool for obtaining a closed formula for the number of spanning trees of  a cograph.
\end{abstract}

{\bf keywords:} cograph, Laplacian eigenvalues, spanning trees. \\


\setlength{\baselineskip}{24pt}
\section{Introduction}
\label{intro}

Let  $G= (V,E)$ be an undirected graph with vertex set $V$ and edge set $E$, without loops or multiple edges.  For $v\in V$, $N(v)$ denotes the  open neighborhood  of $v$, that is, $\{w|\{v,w\}\in E\}$. The  closed neighborhood  $N[v] = N(v) \cup \{v\}$. 
Two vertices  $u,v \in V$ are {\em duplicate}  if $N(u)=N(v)$ and {\em coduplicate} if $N[u]=N[v].$
If $|V| = n$, the   adjacency matrix  $A= (a_{ij})$,  is
 the $n \times n$   matrix of zeros and ones  such that $a_{ij} = 1$ if
there is an edge between $v_i$ and $v_j$, and 0 otherwise.   

The degree sequence of a graph $G$ of order $n$ is the sequence $\delta(v_1), \ldots, \delta(v_n)$ where $\delta(v_i)$ is the degree of vertex $v_i.$   Let $\delta(G)$ be the diagonal matrix of vertex degrees of $G.$ The Laplacian matrix of $G$ is defined as $L(G) = \delta(G) - A(G).$ The Laplacian eigenvalues of $G$ are the eigenvalues of its Laplacian matrix. The Laplacian spectrum of $G$ consists of eigenvalues of $L(G)$ and we write $Spect_{L}(G)=\{ \mu_n, \mu_{n-1}, \ldots, \mu_2, \mu_1 =0 \}.$


Cographs are defined as $P_4$-free graphs, that is,  they do not contain no path on four vertices as an induced subgraph. They are also known as {\em complement reducible graphs} \cite{Stewart}. Linear time recognition algorithms for cographs are given in \cite{Paul, Stewart2,Michel}.
Two important subclasses of cographs are known as thresholds graphs and weakly quasi-threshold graphs \cite{Abreu2,Milica, Bapat,JOT2019,Mah95, Merris, Yan} for which there are several linear time recognition algorithms \cite{Le, Golu, nikolo3}.


It is well known that cographs can be represented by rooted trees, and that a lot of structural and spectral properties about a cograph may be obtained from the tree that produces it  (see, for example \cite{Allem2, Allem, 
BSS2011,JTT2016, Jones,tura22, Mousavi, Niko}).  Some of these works  are worth noting, as we describe below.

First, although the paper \cite{JTT2016} present a linear algorithm for locating the  eigenvalues of a cograph, however in practice for finding a specific value, we must call the algorithm more than one time and then use a divide-and-conquer algorithm for getting an efficient aproximation to it.
Second, the paper  \cite{Jones} explores symmetries present in the tree of a cograph for reducing the complexity of computing its eigenvalues,   but it is still necessary to compute the remaining eigenvalues of matrix produced by it.
 Finally, the paper \cite{Niko} uses the tree  of a cograph  to exhibit an algorithmic procedure for obtaining the number of spanning trees of a cograph.

 
 The purpose of this paper is to present
 a linear algorithm for obtaining the Laplacian eigenvalues of a cograph directly from its tree representation. This  approach is based on the previous algorithms, more specifically, it uses some ideas of  paper Jones et al. \cite{Jones} and the Diagonalization algorithm (see Section 2). Furthermore,  it is more efficient than method presented previously, since there is no need to directly  compute the eigenvalues of Laplacian matrix related to this class of graph.
 Our paper is similar in spirit, but using different proprieties of the class of cographs and the associated trees, to the paper of Bapat  \cite{Bapat} which uses a kind of tree for determining the integers Laplacian eigenvalues of a weakly quasi-threshold graphs.
 As an application, we use the algorithm proposed here as a tool for obtaining a closed formula for the number of spanning trees of  a cograph.

Here is an outline of this paper.
In Section \ref{Sec2}, we provide definitions and known results needed for the development of our paper. 
In Section \ref{Sec3}, we present a linear algorithm for obtaining the Laplacian eigenvalues of a cograph directly from its tree representation.
In Section \ref{Sec4}, we present a closed formula for obtaining the number of spanning trees of a cograph.


\section{Background results}
\label{Sec2}

\subsection{Cographs and cotrees}

In what follows, $G$ denotes a graph with $n$ vertices, while that $\overline{G}$ its complement. 
As usual, $K_n$ and  $\overline{K}_n$ represent the complete graph and the edgeless graph of order $n,$ respectively.
Now, we recall   the definitions of some operations  with graphs that will be used. For this, 
let $G_1= (V_1, E_1)$ and $G_2= (V_2, E_2)$  be  vertex disjoint graphs:
\begin{itemize}
\item The {\em union} of graphs $G_1$ and $G_2$  is the graph $G_1 \cup G_2$  whose vertex set is $V_1 \cup V_2$ and whose edge set is $E_1 \cup E_2.$
\item The {\em join}   of graphs $G_1$ and $G_2$  is the graph $G_1 \otimes G_2$ obtained from $G_1 \cup G_2$ by joining every vertex of $G_1$
with every vertex of $G_2.$
\end{itemize}


 We note that for any graph $G$ on $n$ vertices with $Spect_{L}(G)=\{ \mu_n,  \ldots, \mu_2, \mu_1 =0 \},$ its largest Laplacian eigenvalue $\mu_n(G),$ satisfies $\mu_n(G) \leq n,$ with equality holding if and only if $G$ is a join of two graphs.  Finally, if $\mu_i(G)$ is a Laplacian eigenvalue of $G$ on $n$ vertices then $n -\mu_i(G)$ is a Laplacian eigenvalue of $\overline{G}.$

The class of cographs can be constructed from a single vertex by joining another cograph and by taking complements which is equivalent to say they are closed under the operations of join and disjoin union.
This characterization allows us to represent a cograph by an unique tree, called a {\em cotree} \cite{Stewart}. 

 A cotree $T_G$ of a cograph $G$ is a rooted tree in which any interior vertex $w_i$ is either of $\cup$-type (corresponds to union) or $\otimes$-type (corresponds to join). The leaves are typeless and represent the vertices of the cograph $G.$  An interior vertex is said to be  terminal, if it has no interior vertex as successor. We say that {\em depth} of the cotree is the number of edges of  the longest path from the root to a leaf. To build a cotree for a connected cograph, we simply place a $\otimes$ at the cotree's root, placing $\cup$ on interior vertices with odd depth, and placing $\otimes$ on interior vertices with even depth.

For any interior vertex $w_i \in T_G,$ we denote by $s_i$ and $t_i,$ respectively,  the non negative integers which represent the number of the  immediate successors and  leaves of $w_i.$    
The  Figure \ref{fig2} shows  a cograph $G$  and its cotree $T_G$  with depth equals  to 3.

\begin{figure}[h!]
       \begin{minipage}[c]{0.3 \linewidth}
\begin{tikzpicture}
  [scale=0.5,auto=left,every node/.style={circle}]
  \node[draw,circle,fill=black,label=left:$v_1$] (1) at (1,6) {};  
\node[draw,circle,fill=black,label=left:$v_2$] (2) at (-2,4) {};  
\node[draw,circle,fill=black,label=left:$v_3$] (3) at (-2,2) {};  
\node[draw,circle,fill=black,label=below:$v_4$] (4) at (-0.5,0) {};  
\node[draw,circle,fill=black,label=below:$v_5$] (5) at (2.5,0) {};  
\node[draw,circle,fill=black,label=right:$v_6$] (6) at (4,2) {};  
\node[draw,circle,fill=black,label=right:$v_7$] (7) at (4,4) {};  

  \path (1) edge node[left]{} (2)
        (3) edge node[below]{} (4)
        (5) edge node[below]{} (6)
        (5) edge node[below]{} (1)
        (5) edge node[below]{} (2)
        (5) edge node[below]{} (3)
        (5) edge node[below]{} (4)
        (6) edge node[below]{} (1)
        (6) edge node[below]{} (2)
        (6) edge node[below]{} (3)
        (6) edge node[below]{} (4)
        (7) edge node[below]{} (1)
        (7) edge node[below]{} (2)
        (7) edge node[below]{} (3)
        (7) edge node[below]{} (4);
\end{tikzpicture}
       \end{minipage}\hfill
       \begin{minipage}[l]{0.4 \linewidth}
\begin{tikzpicture}
 [scale=1,auto=left,every node/.style={circle,scale=0.9}]

  \node[draw,circle,fill=black,label=below:$v_6$] (o) at (3,4) {};
  \node[draw,circle,fill=black,label=below:$v_4$] (n) at (1.8,4) {};
   \node[draw,circle,fill=black,label=below:$v_5$] (q) at (2.2,4) {};

  \node[draw, circle, fill=blue!10, inner sep=0] (m) at (2.5,5) {$\otimes$};
  \node[draw,circle,fill=black,label=below:$v_7$] (l) at (4,5) {};

  \node[draw, circle, fill=blue!10, inner sep=0] (j) at (3,6) {$\cup$};
  \node[draw,circle,fill=blue!10, inner sep=0] (h) at (2,7) {$\otimes$};
  \node[draw, circle, fill=blue!10, inner sep=0] (g) at (1,6) {$\cup$};
  \node[draw,circle,fill=blue!10, inner sep=0] (f) at  (1.3,5) {$\otimes$};
  \node[draw,circle,fill=blue!10, inner sep=0] (a) at (0,5) {$\otimes$};
  \node[draw, circle, fill=black, label=below:$v_1$] (b) at (-0.3,4) {};
  \node[draw,circle,fill=black, label=below:$v_2$] (c) at (0.5,4) {};
  \node[draw,circle,fill=black,label=below:$v_3$] (e) at (1,4) {};

  \path (a) edge node[left]{} (b)
        (a) edge node[below]{} (c)

   (f) edge node[below]{} (e)
  (f) edge node[below]{} (n)


        (f) edge node[right]{}(g)
        (g) edge node[left]{}(a)
        (h) edge node[right]{}(j)
        (h) edge node[left]{}(g)

        (j) edge node[right]{}(l)
        (j) edge node[below]{}(m)
        (m) edge node[right]{}(o)
        (m) edge node[left]{} (q);
\end{tikzpicture}
       \end{minipage}
       \caption{A cograph $G=((v_{1}\otimes v_{2})\cup (v_{3}\otimes v_{4}))\otimes ((v_{5}\otimes v_{6}) \cup v_7))$ and its cotree $T_G$.}
       \label{fig2}
\end{figure}

 We note for the complementary cograph $\overline{G}$ its cotree $T_{\overline{G}}$ can be obtained from $T_G$ by changing the type of each respective interior vertex
 and keeping their leaves.  The coduplicate (respect. duplicate) vertices, e.g. with the same neighbors and adjacent (respect. not adjacent) have a common parent of  $ \otimes$-type (respect.  $\cup$-type). Moreover, two vertices $v_i, v_j \in G$ are adjacent if and only
 if their least common ancestor in $T_G$ is $\otimes$-type. We define $lca(w_i, w_j) = \otimes,$ where $w_i$ and $w_j$ are the respective parents of $v_i$ and $v_j$ in $T_G.$ 
From this observation, it is easy to obtain the degree of every vertex of $G,$ just by inspecting its cotree $T_G,$
according to \cite{Jones}.

\begin{Prop}
Let $G$ be a cograph with cotree $T_G$ having $r$ interior vertices $\{w_1,\ldots, w_r\}.$ For a fixed $j \in \{1, \ldots, r\}$ having $t_j \geq 1$ leaves,
the degree $\delta(v_j)$ of vertices $v_j $ that represent the leaves $t_j$  of  $w_j$ is  
$$\delta(v_j)= \left\{
\begin{array}{lr}
(t_j -1) + \sum_i t_i, & \mbox{if  $w_j = \otimes,$ sum being over $i's$ with $lca(w_i, w_j)= \otimes$,} \\
 \sum_i t_i , & \mbox{if $w_j= \cup,$ sum being over $i's$ with $lca(w_i, w_j)=\otimes.$ }
\end{array} \right.$$  

\end{Prop}

\subsection{Diagonalization Algorithm}

An important tool presented in \cite{tura22} was an algorithm
for constructing a {\em diagonal} matrix congruent to $L(G) + \myvar I_n$,
where $L(G)$ is the Laplacian matrix of a cograph $G,$
and $\myvar$ is an arbitrary scalar, using $O(n)$ time and space.

The algorithm's input  is the cotree $T_G$ and $x$. Each leaf $v_i$, $i =1,\ldots,n$ have a value $d_i$ that represents the diagonal element of  $L(G)+xI_n$. It initializes  all entries  with $\delta(v_i) +x,$ where $\delta(v_i)$ denotes the degree of vertex $v_i.$
At each iteration, a pair $\{v_k, v_l\} $ of duplicate or coduplicate vertices with maximum depth  is selected. Then they are processed, that is,  assignments are given to $d_k$ and $d_l$, such that either one or both rows (columns) are diagonalized. When the $k$th row (column)  corresponding to vertex  $v_k$  has been diagonalized then $v_k$ is  removed from the $T_G$, it means that $d_k$ has a permanent final value. Then the algorithm moves to the cotree $T_G -v_k$. The algorithm is shown in Figure \ref{algo}.

\begin{figure}[h]
{\tt
\begin{tabbing}
aaa\=aaa\=aaa\=aaa\=aaa\=aaa\=aaa\=aaa\= \kill
     \> INPUT:  cotree $T_G$, scalar $\myvar$\\
     \> OUTPUT: diagonal matrix $D=[d_1, d_2, \ldots, d_n]$ congruent to $L(G) + \myvar I_n$\\
     \>\\
     \>   $\mbox{ Algorithm}$ Diagonal $(T_{G}, x)$ \\
     \> \> initialize $d_i := \delta(v_i)+ \myvar$, for $ 1 \leq i \leq n$ \\
     \> \> {\bf while } $T_G$  has $\geq 2$    leaves      \\
     \> \> \>  select a pair $(v_k, v_l)$  (co)duplicate of maximum depth with  parent $w$\\
     \> \> \>     $\alpha \leftarrow  d_k$    $\beta \leftarrow d_{l}$\\
     \> \> \> {\bf if} $ w=\otimes$\\
     \> \> \> \>  {\bf if} $\alpha + \beta \neq -2$  \verb+                //subcase 1a+    \\
     \> \> \> \> \>   $d_{l} \leftarrow \frac{\alpha \beta -1}{\alpha + \beta +2};$ \hspace*{0,25cm} $d_{k} \leftarrow \alpha + \beta +2; $\hspace{0,25cm}   $T_G = T_G - v_k$ \\
     \> \> \> \>  {\bf else if } $\beta=-1$ \verb+                //subcase 1b+   \\
     \> \> \> \> \>   $d_{l} \leftarrow -1$ \hspace*{0,25cm}   $d_k  \leftarrow 0;$ \hspace{0,25cm} $T_G = T_G - v_k$ \\
     \> \> \> \>  {\bf else  }  \verb+                      //subcase 1c+   \\
     \> \> \> \> \>   $d_{l} \leftarrow -1$  \hspace*{0,25cm} $d_k \leftarrow (1+\beta)^2;$ \hspace{0,25cm} $T_G= T_G -v_k;$ \hspace{0,25cm} $T_G = T_G -v_l$  \\
     \> \> \>     {\bf else if} $w=\cup$\\
     \> \> \> \>  {\bf if} $\alpha + \beta \neq 0$  \verb+               //subcase 2a+    \\
     \> \> \> \> \>   $d_{l} \leftarrow \frac{\alpha \beta}{\alpha +\beta};$ \hspace*{0,25cm}   $d_k \leftarrow \alpha +\beta;$ \hspace{0,25cm} $T_G = T_G - v_k$ \\
     \> \> \> \>  {\bf else if } $\beta=0$ \verb+                //subcase 2b+   \\
     \> \> \> \> \>   $d_{l} \leftarrow 0;$ \hspace*{0,25cm}  $d_k  \leftarrow 0;$ \hspace{0,25cm} $T_G = T_G - v_k$ \\
     \> \> \> \>  {\bf else  }  \verb+                      //subcase 2c+   \\
     \> \> \> \> \>   $d_{l} \leftarrow \beta;$  \hspace*{0,25cm} $v_k \leftarrow -\beta;$ \hspace{0,25cm} $T_G =T_G - v_k;$ \hspace{0,25cm} $T_G = T_G - v_l$  \\
     \> \>  {\bf end loop}\\
\end{tabbing}
}
\caption{\label{algo} Diagonalization algorithm}
\end{figure}

The next  result from \cite{tura22} will be used throughout the paper.

\begin{Thr}
\label{main1}
For a cograph $G$ of order $n$ with cotree $T_G,$ let  $D=[d_1,d_2,\ldots,d_n]$ be the values   produced  by the Diagonalization Algorithm $(T_G,-x).$ Then the diagonal matrix $D$ is congruente to $L(G)+\myvar I_n,$ hence the number of ( positive $|$ negative $|$ zero) entries in $D$ is equal to the number of Laplacian eigenvalues of $L(G)$ that are ( greater than $\myvar$ $|$ smaller than $\myvar$ $|$ equal to $\myvar$).
\end{Thr}

\begin{Ex} We will apply  Diagonalization to $G$ with $x=-7,$ having cotree $T_G$ represented in Figure \ref{fig3a}.
In our figures, diagonal values $d_i$ appear under the vertex $v_i$ in the cotree while vertices selected appear in red. After initialization, we have $T_G$ with values $d_i= \delta(v_i)+x$, represented in Figure \ref{fig3a}.

\begin{figure}[h!]
\begin{tikzpicture}
   [scale=1,auto=left,every node/.style={circle,scale=0.9}]
  \node[draw,circle,fill=black,label=below:$0$] (o) at (2,5) {};
  
  \node[draw,circle,fill=black, label=below:$-2$] (n) at (1.5,5) {};
   \node[draw,circle,fill=black, label=below:$-2$] (n1) at (1,5) {};
  \node[draw,circle,fill=black, label=below:$-2$] (e) at (0.5,5) {};

  \node[draw,circle,fill=black, label=below:$0$] (l) at (3,5) {};
 \node[draw,circle,fill=black, label=below:$0$] (v) at (2.5,5) {};
 \node[draw,circle,fill=black, label=below:$0$] (v1) at (3.5,5) {};
 \node[draw,circle,fill=black, label=below:$0$] (v2) at (4,5) {};
  \node[draw, circle, fill=blue!10, inner sep=0, label=right:$$] (j) at (3,6) {$\cup$};
  \node[draw,circle,fill=blue!10, inner sep=0, label=left:$$] (h) at (2,7) {$\otimes$};

  \node[draw, circle, fill=blue!10, inner sep=0, label=left:$$] (g) at (1,6) {$\cup$};
    
  \node[draw,circle,fill=blue!10, inner sep=0,label=left:$$] (a) at (0,5) {$\otimes$};
  
  \node[draw, circle, fill=blue!10, inner sep=0, label=right:$$] (g1) at (0.25,4) {$\cup$};
  \node[draw, circle, fill=black, label=below:$0$] (b) at (0.25,3) {};
  \node[draw,circle,fill=black, label=below:$0$] (c) at (0.75,3) {};
  
   \node[draw, circle, fill=blue!10, inner sep=0, label=left:$$] (g2) at (-0.5,4) {$\cup$};
 \node[draw,circle,fill=black,label=below:$0$] (e1) at (-1,3) {};
  \node[draw,circle,fill=black,label=below:$0$] (e2) at (-0.25,3) {};

  \path 
        (a) edge node[below]{} (g2)
(e1) edge node[below]{} (g2)
(e2) edge node[below]{} (g2)

(a) edge node[left]{} (g1)
(g1) edge node[left]{} (b)
(g1) edge node[left]{} (c)

  (g) edge node[below]{} (e)
  (g) edge node[below]{} (n)
 (g) edge node[below]{} (n1)

    (j) edge node[below]{} (v)
    (j) edge node[below]{} (v1)
(j) edge node[below]{} (v2)      
        (g) edge node[left]{}(a)
        (h) edge node[right]{}(j)
        (h) edge node[left]{}(g)

        (j) edge node[right]{}(l)
       (j) edge node[below]{}(o);
\end{tikzpicture}
       \caption{ Initial  $T_G$ with $d_i= \delta(v_i) +x.$}
       \label{fig3a}
\end{figure}

In the first iteration of the algorithm, duplicate vertices with maximum depth  are chosen, represented in  Figure \ref{fig4a}.

\begin{figure}[h!]
       \begin{minipage}[c]{0.45 \linewidth}
\begin{tikzpicture}
   [scale=1,auto=left,every node/.style={circle,scale=0.9}]
  \node[draw,circle,fill=black,label=below:$-2$] (o) at (2,5) {};
  
  \node[draw,circle,fill=black, label=below:$-2$] (n) at (1.5,5) {};
   \node[draw,circle,fill=black, label=below:$-2$] (n1) at (1,5) {};

  \node[draw,circle,fill=black, label=below:$0$] (l) at (3,5) {};
 \node[draw,circle,fill=black, label=below:$0$] (v) at (2.5,5) {};
 \node[draw,circle,fill=black, label=below:$0$] (v1) at (3.5,5) {};
 \node[draw,circle,fill=black, label=below:$0$] (v2) at (4,5) {};
  \node[draw, circle, fill=blue!10, inner sep=0, label=right:$$] (j) at (3,6) {$\cup$};
  \node[draw,circle,fill=blue!10, inner sep=0, label=left:$$] (h) at (2,7) {$\otimes$};

  \node[draw, circle, fill=blue!10, inner sep=0, label=left:$$] (g) at (1,6) {$\cup$};
    
  \node[draw,circle,fill=blue!10, inner sep=0, label=left:$$] (a) at (0,5) {$\otimes$};
  
  \node[draw, circle, fill=blue!10, inner sep=0, label=right:$$] (g1) at (0.25,4) {$\cup$};
  \node[draw, circle, fill=red, label=below:$0$] (b) at (0.25,3) {};
  \node[draw,circle,fill=red, label=below:$0$] (c) at (0.75,3) {};
  \node[draw,circle,fill=black, label=below:$-2$] (e) at (0.5,5) {};
  
   \node[draw, circle, fill=blue!10, inner sep=0,label=left:$$] (g2) at (-0.5,4) {$\cup$};
 \node[draw,circle,fill=red,label=below:$0$] (e1) at (-1,3) {};
  \node[draw,circle,fill=red,label=below:$0$] (e2) at (-0.25,3) {};

  \path 
        (a) edge node[below]{} (g2)
(e1) edge node[below]{} (g2)
(e2) edge node[below]{} (g2)

(a) edge node[left]{} (g1)
(g1) edge node[left]{} (b)
(g1) edge node[left]{} (c)

  (g) edge node[below]{} (e)
  (g) edge node[below]{} (n)
 (g) edge node[below]{} (n1)

    (j) edge node[below]{} (v)
    (j) edge node[below]{} (v1)
(j) edge node[below]{} (v2)      
        (g) edge node[left]{}(a)
        (h) edge node[right]{}(j)
        (h) edge node[left]{}(g)

        (j) edge node[right]{}(l)
       (j) edge node[below]{}(o);
\end{tikzpicture}
\caption{ First iteration}
       \label{fig4a}
       \end{minipage}\hfill
       \begin{minipage}[c]{0.4 \linewidth}
\begin{tikzpicture}
   [scale=1,auto=left,every node/.style={circle,scale=0.9}]

   [scale=1,auto=left,every node/.style={circle,scale=0.9}]
  \node[draw,circle,fill=black,label=below:$0$] (o) at (2,5) {};
  
  \node[draw,circle,fill=black, label=below:$-2$] (n) at (1.5,5) {};
   \node[draw,circle,fill=black, label=below:$-2$] (n1) at (1,5) {};

  \node[draw,circle,fill=black, label=below:$0$] (l) at (3,5) {};
 \node[draw,circle,fill=black, label=below:$0$] (v) at (2.5,5) {};
 \node[draw,circle,fill=black, label=below:$0$] (v1) at (3.5,5) {};
 \node[draw,circle,fill=black, label=below:$0$] (v2) at (4,5) {};
  \node[draw, circle, fill=blue!10, inner sep=0,  label=right:$$] (j) at (3,6) {$\cup$};
  \node[draw,circle,fill=blue!10, inner sep=0,  label=left:$$] (h) at (2,7) {$\otimes$};

  \node[draw, circle, fill=blue!10, inner sep=0, label=left:$$] (g) at (1,6) {$\cup$};
    
  \node[draw,circle,fill=blue!10, inner sep=0,label=left:$$] (a) at (0,5) {$\otimes$};
  
  \node[draw, circle, fill=red,label=below:$0$] (g1) at (0.25,4) {};
  \node[draw, circle, fill=black, label=below:$0$] (b) at (0.25,3) {};
  \node[draw,circle,fill=black, label=below:$-2$] (e) at (0.5,5) {};
  
   \node[draw, circle, fill=red,label=below:$0$] (g2) at (-0.5,4) {};
 \node[draw,circle,fill=black,label=below:$0$] (e1) at (-1,3) {};

  \path 
        (a) edge node[below]{} (g2)

(a) edge node[left]{} (g1)

  (g) edge node[below]{} (e)
  (g) edge node[below]{} (n)
 (g) edge node[below]{} (n1)

    (j) edge node[below]{} (v)
    (j) edge node[below]{} (v1)
(j) edge node[below]{} (v2)      
        (g) edge node[left]{}(a)
        (h) edge node[right]{}(j)
        (h) edge node[left]{}(g)

        (j) edge node[right]{}(l)
       (j) edge node[below]{}(o);\end{tikzpicture}
       \caption{After \subiib~applied}
       \label{fig5a}
        \end{minipage}
\end{figure}

Since $\beta=0$ the \subiib~is executed, it means, two vertices are removed from cotree with a permanent null value 
and two remaining vertices with value $0$ are relocated under its parent's, represented in Figure \ref{fig5a}.

In the next step, coduplicate vertices with maximum depth are selected, represented in Figure \ref{fig5a}.
Since $\alpha +\beta = 0+0 \neq -2,$  the \subia~is executed. Then one vertex is removed from cotree with a value $2$
while another is relocated under its parent's with value $-\frac{1}{2},$ represented in Figure \ref{fig6a}.

\begin{figure}[h!]
       \begin{minipage}[c]{0.45 \linewidth}

\begin{tikzpicture}
[scale=1,auto=left,every node/.style={circle,scale=0.9}]

  \node[draw,circle,fill=black,label=below:$0$] (o) at (2,5) {};
  
  \node[draw,circle,fill=red, label=below:$-2$] (n) at (1.5,5) {};
   \node[draw,circle,fill=red, label=below:$-2$] (n1) at (1,5) {};

  \node[draw,circle,fill=black, label=below:$0$] (l) at (3,5) {};
 \node[draw,circle,fill=black, label=below:$0$] (v) at (2.5,5) {};
 \node[draw,circle,fill=black, label=below:$0$] (v1) at (3.5,5) {};
 \node[draw,circle,fill=black, label=below:$0$] (v2) at (4,5) {};
  \node[draw, circle, fill=blue!10, inner sep=0,  label=right:$$] (j) at (3,6) {$\cup$};
  \node[draw,circle,fill=blue!10, inner sep=0,  label=left:$$] (h) at (2,7) {$\otimes$};

  \node[draw, circle, fill=blue!10, inner sep=0,label=left:$$] (g) at (1,6) {$\cup$};
    
  \node[draw,circle,fill=red,label=below:$-\frac{1}{2}$] (a) at (0,5) {};
  
  \node[draw, circle, fill=black, label=below:$0$] (b) at (0,3) {};
  \node[draw,circle,fill=red, label=below:$-2$] (e) at (0.5,5) {};
  
   \node[draw, circle, fill=black,label=below:$2$] (g2) at (1,3) {};
 \node[draw,circle,fill=black,label=below:$0$] (e1) at (-1,3) {};

  \path 


  (g) edge node[below]{} (e)
  (g) edge node[below]{} (n)
 (g) edge node[below]{} (n1)

    (j) edge node[below]{} (v)
    (j) edge node[below]{} (v1)
(j) edge node[below]{} (v2)      
        (g) edge node[left]{}(a)
        (h) edge node[right]{}(j)
        (h) edge node[left]{}(g)

        (j) edge node[right]{}(l)
       (j) edge node[below]{}(o);\end{tikzpicture}
\caption{ After \subia~applied}
       \label{fig6a}
       \end{minipage}\hfill
       \begin{minipage}[c]{0.4 \linewidth}\begin{tikzpicture}
 [scale=1,auto=left,every node/.style={circle,scale=0.9}]

  \node[draw,circle,fill=red,label=below:$0$] (o) at (2,5) {};
  
  \node[draw,circle,fill=black, label=below:$-4$] (n) at (1.25,3) {};
   \node[draw,circle,fill=black, label=below:$-3$] (n1) at (2,3) {};

  \node[draw,circle,fill=red, label=below:$0$] (l) at (3,5) {};
 \node[draw,circle,fill=red, label=below:$0$] (v) at (2.5,5) {};
 \node[draw,circle,fill=red, label=below:$0$] (v1) at (3.5,5) {};
 \node[draw,circle,fill=red, label=below:$0$] (v2) at (4,5) {};
  \node[draw, circle, fill=blue!10, inner sep=0,label=right:$$] (j) at (3,6) {$\cup$};
  \node[draw,circle,fill=blue!10, inner sep=0,  label=left:$$] (h) at (2,7) {$\otimes$};

  \node[draw, circle, fill=black,label=left:$-\frac{2}{7}$] (g) at (1,6) {};
    
  \node[draw,circle,fill=black,label=below:$-\frac{7}{6}$] (a) at (2.75,3) {};
  
  \node[draw, circle, fill=black, label=below:$0$] (b) at (-0.25,3) {};
  
   \node[draw, circle, fill=black,label=below:$2$] (g2) at (0.5,3) {};
 \node[draw,circle,fill=black,label=below:$0$] (e1) at (-1,3) {};

  \path 



    (j) edge node[below]{} (v)
    (j) edge node[below]{} (v1)
(j) edge node[below]{} (v2)      
        (h) edge node[right]{}(j)
        (h) edge node[left]{}(g)

        (j) edge node[right]{}(l)
       (j) edge node[below]{}(o);\end{tikzpicture}
       \caption{After \subiia~applied}
       \label{fig7a}
        \end{minipage}
\end{figure}

For the duplicate vertices on the left of cotree, we apply successive times the \subiia. 
The removed vertices that have been processed from right to left have assignments equal to $-4, -3$ and $-\frac{7}{6},$ while the vertex relocated under its parent's has assignment $-\frac{2}{7}.$
represented in Figure \ref{fig7a}.

\begin{figure}[h!]
\begin{tikzpicture}
 [scale=1,auto=left,every node/.style={circle,scale=0.9}]

  
  \node[draw,circle,fill=black, label=below:$-4$] (n) at (1.25,4) {};
   \node[draw,circle,fill=black, label=below:$-3$] (n1) at (2,4) {};

  \node[draw,circle,fill=black, label=below:$0$] (l) at (5.75,4) {};
 \node[draw,circle,fill=black, label=below:$0$] (v) at (5,4) {};
 \node[draw,circle,fill=black, label=below:$0$] (v1) at (3.5,4) {};
 \node[draw,circle,fill=black, label=below:$0$] (v2) at (4.25,4) {};
  \node[draw, circle, fill=red,label=right:$0$] (j) at (3,6) {};
  \node[draw,circle,fill=blue!10, inner sep=0, label=right:$$] (h) at (2,7) {$\otimes$};

  \node[draw, circle, fill=red,label=left:$-\frac{2}{7}$] (g) at (1,6) {};
    
  \node[draw,circle,fill=black,label=below:$-\frac{7}{6}$] (a) at (2.75,4) {};
  
  \node[draw, circle, fill=black, label=below:$0$] (b) at (-0.25,4) {};
  
   \node[draw, circle, fill=black,label=below:$2$] (g2) at (0.5,4) {};
 \node[draw,circle,fill=black,label=below:$0$] (e1) at (-1,4) {};

  \path 



        (h) edge node[right]{}(j)
        (h) edge node[left]{}(g);


\end{tikzpicture}
       \caption{ Last iteration}
       \label{fig8a}
\end{figure}

For the last duplicate vertices since $\beta=0$ the \subiia~is executed again.
The removed vertices from cotree have a permanent null value while the remaining vertex has value $0$ is relocated under its parent's, represented in Figure \ref{fig8a}.

Finally, for the last iteration, we have coduplicate vertices with assignments $-\frac{2}{7}$ and $0.$ Since $\alpha+\beta\neq -2$
the \subia~is executed and the assignments $-\frac{7}{12}$ and $\frac{12}{7}$ are given. Therefore, the final diagonal is $D=[0,0,2,-4,-3,-\frac{7}{6},0,0,0,0,-\frac{7}{12},\frac{12}{7}].$

According to Theorem \ref{main1}, there are two Laplacian eigenvalues greather than $7,$ four Laplacian eigenvalues  less than $7$
and $7$ is a Laplacian eigenvalue with multiplicity 6. In fact, we have that  $Spect_{L}(G)=\{12,9,7^6,5^3,0\}.$
\end{Ex}


\begin{Lem}
\label{Lem1}
Let $G$ be a  cograph with cotree $T_G.$  Let $t_i \geq 2$ be the leaves of an interior vertex $w_i$ of $T_G.$ Initializing the Diagonalization Algorithm with $x=-\delta(v_i)$ (respect. $x=-\delta(v_i)-1$) for duplicate (respect. coduplicate) vertices  $v_i$ then  $\delta(v_i)$ (respect. $\delta(v_i) +1$)  is a Laplacian eigenvalue  with multiplicity $t_i -1,$ if $w_i=\cup$-type (respect.  $w_i=\otimes$-type). 
\end{Lem}

\begin{proof}
Let $t_i \geq 2$ be the leaves with degree $\delta(v_i)$ of an interior vertex  $w_i$ of $T_G.$
Initializing the Diagonalization algorithm with $x=-\delta(v_i)$ (respect.  $x=-\delta(v_i)-1$) for duplicate (respect. coduplicate) vertices  $v_i$
and since $\beta =0$ (respect. $\beta=-1)$
the algorithm enters to the \subiib~(respect. \subib). Therefore, it assigns $t_i -1$ permanent  values zero.
\end{proof}

\begin{Lem}
\label{Lem2}
Let $G$ be a connected cograph with cotree $T_{G}.$ The  Diagonalization $(T_G,x=0)$  assigns a permanent null  value  in the last iteration by \subia.
\end{Lem}

\begin{proof} 
Let $G$ be a  connected cograph with  cotree $T_{G}.$
Let $x=0$ be a Laplacian eigenvalue of $G$ and consider the execution of Diagonalization $(T_G,x=0).$
According Theorem \ref{main1}, we must have a zero on the final diagonal. 
So, the permanent null value is given by either situations:
 \subib,~
 \subiib,~ 
or \subia~in the last iteration. 

We assume that  \subib~or \subiib~are executed in an intermediate step of Diagonalization $(T_G,x=0).$ 
Let's denote by $T_{G_i}$ the subcotree which has assigned this permanent zero. Let $H$ be a connected cograph with a cotree $T_H$ obtained from $T_G$ by attaching a copy of $ T_{G_i}$ and consider the Diagonalization $(T_H,x=0).$ Since we will have two permanent null values, this is a contradiction.  Therefore, follows that permanent  null  value  is assigned by \subia~ in the last iteration. 
\end{proof}


\section{ A Linear Algorithm for obtaining $Spect_{L}(G)$}
\label{Sec3}

In this section, we present an $O(n)$ time and space algorithm for computing the Laplacian eigenvalues of a cograph $G.$ Our approach uses the cotree $T_G$ of  cograph $G$ to perform this procedure.

The algorithm's input is a cotree $T_G$ with  vertices (including the interior vertices) having a weight.
It initializes with  all leaves  having a weight $p_i=1$ and the interior vertices $w_i$ with degree $\delta(w_i)$ (see Definition \ref{def1}). 
At each iteration, a set  $\{v_1, \ldots, v_m\}$ of duplicate or coduplicate vertices with maximum depth  is selected. Then they are processed, that is,  assignments are given to them.  The assignments given to $v_1, \ldots, v_{m-1}$  corresponds to a Laplacian eigenvalue of $G$ while the vertex $v_m$ receives the value $\sum_i p_i.$ Then $v_1, \ldots, v_{m-1}$  are removed from the $T_G$, it means that $v_1, \ldots, v_{m-1}$ have a permanent final value. The vertex $v_m$ is relocated under its parent's and then the algorithm moves to the cotree $T_G -\{v_1, \ldots, v_{m-1}\}.$
This process is repeated until we obtain a cotree with only one leaf.
We consider the following definition.

\begin{Def}
\label{def1}
Let $G$ be a cograph with cotree $T_G$  having $r$ interior vertices $\{w_1, \ldots, w_r\}.$  
We assume that  leaves of an interior vertex $w_i$ have vertices $v_i$  with a weight $p_i.$ For a fixed $j  \in \{1, \ldots, r\},$  the degree of an interior vertex $w_j,$ denoted by $\delta(w_j),$  is given by 
$$\delta(w_j) = \sum_{i} p_i$$
sum being over $i's$ with $lca(w_i, w_j)= \otimes.$
\end{Def}


\begin{Prop}
\label{prop2}
Let $G$ be a cograph of order $n$ with cotree $T_G$ having $r$ interior vertices $\{w_1, \ldots, w_r\}.$ We assume that  leaves $t_i\geq 1$ of an interior vertex $w_i$ have vertices $v_i$ with a weight $p_i=1.$ For any interior vertex $w_j,$ $j\in \{1,\ldots, r\},$ we have
$$\delta(w_j)= \left\{
\begin{array}{lr}
t_j  + \sum_i t_i, & \mbox{if  $w_j = \otimes,$ sum being over $i's$ with $lca(w_i, w_j)= \otimes$,} \\
 \sum_i t_i , & \mbox{if $w_j= \cup,$ sum being over $i's$ with $lca(w_i, w_j)=\otimes.$ }
\end{array} \right.$$  
Furthermore, for $j\in \{1,\ldots, r\},$ holds 
$$ n = \delta(w_j) + \delta(\overline{w_j}),$$
where $\overline{w_j}$ is the respective interior vertex of $w_j$ with opposite type in the cotree $T_{\overline{G}}.$
\end{Prop}

\begin{proof} 
Let $G$ be a cograph of order $n$ with cotree $T_G$ having $r$ interior vertices $\{w_1, \ldots, w_r\}.$  We assume that leaves $t_i\geq 1$ of an interior vertex $w_i$  have vertices $v_i$ with a weight $p_i=1.$ Let $w_j$ be an interior vertex with $\delta(w_j).$ By definition, we have that
$ \delta(w_j) = \sum_i p_i$ sum being over $i's$ with $lca(w_i, w_j)= \otimes.$ So, if $w_j=\cup$ then  $\delta(w_j)$ is equivalent to sum the $t_i,$  which have with $w_j$ a least common ancestor of $\otimes$-type, since the leaves of cotree have weight $p_i=1.$
Otherwise, if $w_j=\otimes$ and it has $t_j\geq 1$ leaves with a weight $p_j=1$ then $\sum_j t_j$  must be included in $\delta(w_j)$ since the $lca(w_j, w_j)= w_j.$

In relation to the second statement, we first note that $\delta(w_j) \leq n,$ for any interior vertex $w_j \in T_G.$  Let $p\geq 0 $ be an integer such that $n=p+\delta(w_j).$ Since $p$ represents the number of leaves $t_i$ of $w_i \in T_G$ with $lca(w_i, w_j)=\cup,$ including the $t_j$ leaves of $w_j,$ if $w_j =\cup$-type, then  $p = \delta(\overline{w_j})$ where $\overline{w_j}$ is the respective interior vertex of $w_j$ with opposite type in the cotree $T_{\overline{G}}$ and therefore, follows the result.
\end{proof}


\begin{Lem}
\label{Lem3}
Let $G$ be a cograph  of order $n$ with  cotree  $T_{G}.$ We assume that the leaves of $T_G$ have initialized with weight $p_i=1$ and the interior vertices $w_i$
with  degree $\delta(w_i).$  If $w_j$ is an interior vertex with maximum depth  in $T_G$ with $t_j \geq 2 $ leaves having vertices with weight $p_j$
then $\delta(w_j)$ is a Laplacian eigenvalue of $G$ with mutiplicity $t_j -1.$ 
\end{Lem}

\begin{proof} 
 Let $G$ be a cograph  of order $n$ with cotree $T_G.$ We assume that  leaves of $T_G$ have initialized with  weight $p_i=1$ and the interior vertices with degree $\delta(w_i).$ 
 Let  $w_j $ be an interior vertex  with maximum depth in $T_G$ with $t_j\geq 2$ leaves having vertices with weight $p_j.$ We claim that $\delta(w_j)$ is a Laplacian eigenvalue of $G$  with multiplicity $t_j-1.$ For getting the result  it  suffices to show that Diagonalization $(T_G, x=-\delta(w_j))$ assigns $t_j -1$ permanent zeros after processing the $t_j $ leaves.
 We consider the following cases:\\
\noindent{\em Case 1.} If $w_j = \cup.$  If $w_j$ is a terminal vertex of $T_G,$ in this case, we have that  $\delta(w_j)$ coincides with the degree of vertices in the  leaves $t_j$ according to Proposition \ref{prop2} and hence the result follows according to Lemma \ref{Lem1}. If  $w_j$ is not a terminal vertex of $T_G,$  we have that some leaves have been processed previously. Since duplicate vertices assign a permanent value zero  if and only if the \subiib~occurs then  duplicate vertices must received the value zero from coduplicate vertices during their last iteration.
Let us denote $T_{H_i}$  the subcotree  of $T_G$ with  $w_i$ has been  an immediate successor of $w_j$  having  $t_i$ leaves, such that, after to  process their leaves $t_i$  has a pendant vertex with  assignment zero.
Taking into account that leaves of  $T_{H_i}$ have initialized with $\delta(t_i) +\delta(w_j)-\delta(w_j),$ where $\delta(t_i)$ is the degree of leaves  only in the subcotree $T_{H_i}$  and this  is equivalent to apply the Diagonalization  $(T_{H_i},x=0).$ So, by Lemma \ref{Lem2}, we will have $t_j$ duplicate vertices with assigment equal to zero. Therefore, $t_j -1$ zeros will be  assigned, as desired.

\noindent{\em Case 2.} If $w_j = \otimes.$ Consider the complement graph  $\overline{G}$ and its cotree $T_{\overline{G}}.$
Let $\overline{w_j}$ be the respective vertex of $w_j$ with opposite type in $T_{\overline{G}}$ having  $t_j$ leaves.
 By previous case, $\delta(\overline{w_j})$  is a Laplacian eigenvalue of $\overline{G}$ with multiplicity $t_j-1.$  By Proposition \ref{prop2}, we have  $n- \delta(\overline{w_j}) = \delta(w_j),$ and therefore, $\delta(w_j)$ is a Laplacian eigenvalue of $G$ with multiplicity $t_j -1,$ as desired.  
 \end{proof}

\begin{figure}[t]
{\tt
\begin{tabbing}
aaa\=aaa\=aaa\=aaa\=aaa\=aaa\=aaa\=aaa\= \kill
\> INPUT:  cotree $T_{G}$ with values $p_i$ and $\delta(w_i)$\\
     \> OUTPUT:   $L$-eingenvalues  of $G$\\
      \>\\
     \>   $\mbox{ Algorithm L-eingenvalues}$  \\ 
     \>\> initialize   $p_i =1$ \\
    \> \> {\bf while } $T_{G}$  has  depth $\geq 2$        \\
     \> \> \>  select a set $(v_1, \ldots, v_m)$  (co)duplicate of maximum depth with  parent $w$\\
     \> \> \>  having $\delta(w)$\\
     \> \> \> {\bf if} $ w=\{\otimes, \cup \}$ \\
     \> \> \> \> \>   $v_1,\ldots,v_{m-1} \leftarrow  \delta(w); \hspace*{0,25cm}  v_{m} \leftarrow  \sum_i p_i; \hspace{0,25cm}   T_{G} = T_{G} - (v_1, \ldots, v_{m-1})$ \\
      \>\>{\bf if}    $T_{G}$  has  depth $1$\\
      \> \> \> {\bf if} $ w=\{\otimes, \cup \}$\\
      \> \> \> \>\>    $v_1, \ldots, v_{m-1} \leftarrow  \delta(w); \hspace{0,25cm} v_m \leftarrow 0;\hspace{0,5cm} T_{G} = T_{G} - (v_1, \ldots, v_{m-1})$  \\ 
     \> \>  {\bf end loop}\\
\end{tabbing}
}
\caption{\label{fig04} Algorithm $L$-eigenvalues}
\end{figure}

The next result claims that the Algorithm $L$-eigenvalues computes the Laplacian eigenvalues of cograph $G$ directly from $T_G$  in $O(n)$ time and space.

Given a weighted cotree $T_G$ where the leaves have initialized with value $p_ i=1$ and the interior vertices with their degrees $\delta(w_i).$ The Algorithm $L$-eigenvalues  performs directly on the leaves of $T_G$ which makes the procedure linear in the number of leaves.
Furthermore, in each iteration it records changes of the values on it such that the removed vertices $(v_1, \ldots, v_{m-1})$ represent a Laplacian eigenvalue of $G$ with multiplicity $m-1$ while that remaning vertex $v_m$ with value $\sum_i p_i$  which preserves the degrees of remaining  interior vertices in $T_G -(v_1, \ldots, v_m),$ allowing only $O(n)$ space.

\begin{Thr}
\label{main2}
Let $G$ be a cograph of order $n$  with cotree $T_G$ having depth $l\geq 1.$  We assume that  leaves $t_i$  of $T_G$ have initialized with  weight $p_i=1$ and the interior vertices with degree $\delta(w_i).$ 
The Laplacian eigenvalues of $G$ can be computed in $O(n)$  by the Algorithm  $L$-eigenvalues, in the Figure \ref{fig04}, and they are given by
\begin{itemize}
\item[(i)] $\delta(w_j)$ with multiplicity $t_j-1,$  if $ l \geq 2, w_j=\{\otimes, \cup\}$ and $t_j \geq 2.$ 
\item[(ii)] $\delta(w_j)$ and $0$ with multiplicities $ t_j-1$ and $1,$ if $l=1, w_j= \{\otimes, \cup \}$ and $t_j \geq 2.$
\end{itemize}
\end{Thr}

\begin{proof} 
Let $G$ be a cograph of order $n$  with cotree $T_G$ having depth $l\geq 1.$  We assume that  leaves $t_i$ of $T_G$ have initialized with  weight $p_i=1$ and the interior vertices with degree $\delta(w_i).$ 
We prove by induction on depth $l\geq 1.$

For $l=1,$ we have to consider two cases. If $w_j = \cup$ with $n=t_j$ leaves having vertices with weight $p_j=1.$
Since $\delta(w_j) = 0,$ we have that $0$ is a Laplacian eigenvalue with multiplicity $n=t_j.$ If 
$w_j = \otimes$ with $n=t_j$ leaves having vertices with weight $p_j=1.$ Since $\delta(w_j)=\sum_j p_j =n$ then
$n$ and $0$ are the Laplacian eigenvalues of $G$ with multiplicity $t_j-1=n-1$ and $1,$ respectively.

Now, we suppose that  Algorithm $L$-eigenvalues holds for a cograph $G$ with
 cotree $T_G$ having  depth $\leq l.$
Let $G'$ be a cograph with cotree $T_{G'}$ having depth equals to $l+1.$  
If $l+1$ is even (respect. odd), then $T_{G'}$ have $t_j$ leaves of an interior vertex $w_j = \cup$ (respect. $w_j=\otimes$) which are duplicate (respect. coduplicate) vertices. By Lemma \ref{Lem3}, $\delta{(w_j)}$ is a Laplacian eigenvalue with multiplicity
$t_j-1.$  Since the pending vertex is relocated under its parent's, the algorithm moves to the cotree  which has a depth equals to $l.$ Then applying the induction on this cotree,
this completes the proof of Theorem.
\end{proof}


\begin{Ex} We consider the cograph $G=(( \overline{K}_2 \otimes \overline{K}_2)\cup \overline{K}_3)\otimes \overline{K}_5$ with cotree $T_G$  represented in  Figure \ref{fig4}.
We will show that $Spect_{L}(G)=\{12,9,7^6,5^3,0\}.$

The Algorithm $L$-eigenvalues initializes the leaves of $T_G$ with $p_i=1$ and the interior vertices with $\delta(w_i).$ During in each iteration a set  of duplicate or coduplicate vertices selected appear in blue.

\begin{figure}[h!]
\begin{tikzpicture}
   [scale=1,auto=left,every node/.style={circle,scale=0.9}]
  \node[draw,circle,fill=black,label=below:$1$] (o) at (2,5) {};
  
  \node[draw,circle,fill=black, label=below:$1$] (n) at (1.5,5) {};
   \node[draw,circle,fill=black, label=below:$1$] (n1) at (1,5) {};

  \node[draw,circle,fill=black, label=below:$1$] (l) at (3,5) {};
 \node[draw,circle,fill=black, label=below:$1$] (v) at (2.5,5) {};
 \node[draw,circle,fill=black, label=below:$1$] (v1) at (3.5,5) {};
 \node[draw,circle,fill=black, label=below:$1$] (v2) at (4,5) {};
  \node[draw, circle, fill=blue!10, inner sep=0, label=right:$7$] (j) at (3,6) {$\cup$};
  \node[draw,circle,fill=blue!10, inner sep=0, label=left:$12$] (h) at (2,7) {$\otimes$};

  \node[draw, circle, fill=blue!10, inner sep=0, label=left:$5$] (g) at (1,6) {$\cup$};
    
  \node[draw,circle,fill=blue!10, inner sep=0,label=left:$9$] (a) at (0,5) {$\otimes$};
  
  \node[draw, circle, fill=blue!10, inner sep=0, label=right:$7$] (g1) at (0.25,4) {$\cup$};
  \node[draw, circle, fill=black, label=below:$1$] (b) at (0.25,3) {};
  \node[draw,circle,fill=black, label=below:$1$] (c) at (0.75,3) {};
  \node[draw,circle,fill=black, label=below:$1$] (e) at (0.5,5) {};
  
   \node[draw, circle, fill=blue!10, inner sep=0, label=left:$7$] (g2) at (-0.5,4) {$\cup$};
 \node[draw,circle,fill=black,label=below:$1$] (e1) at (-1,3) {};
  \node[draw,circle,fill=black,label=below:$1$] (e2) at (-0.25,3) {};

  \path 
        (a) edge node[below]{} (g2)
(e1) edge node[below]{} (g2)
(e2) edge node[below]{} (g2)

(a) edge node[left]{} (g1)
(g1) edge node[left]{} (b)
(g1) edge node[left]{} (c)

  (g) edge node[below]{} (e)
  (g) edge node[below]{} (n)
 (g) edge node[below]{} (n1)

    (j) edge node[below]{} (v)
    (j) edge node[below]{} (v1)
(j) edge node[below]{} (v2)      
        (g) edge node[left]{}(a)
        (h) edge node[right]{}(j)
        (h) edge node[left]{}(g)

        (j) edge node[right]{}(l)
       (j) edge node[below]{}(o);
\end{tikzpicture}
       \caption{ Initial  $T_G$ with $p_i=1$ and $\delta(w_i).$}
       \label{fig4}
\end{figure}

In the first iteration of the algorithm, duplicate vertices $\{v_1, v_2\}$ are chosen of $w_i$ with maximum depth, represented in  Figure \ref{fig5}.
Since $\delta(w_i)=7$  the assignments are given
$$ v_1 \leftarrow 7 ;\hspace{1cm}v_2 \leftarrow 1+1=2; \hspace{1cm} T_{G} = T_{G} - v_1.$$

Then the vertex $v_1$ is removed and $v_2$ is relocated under its parent's, represented in Figure \ref{fig6}.
It means that $7$ is a Laplacian eigenvalue with multiplicity one.

\begin{figure}[h!]
       \begin{minipage}[c]{0.45 \linewidth}
\begin{tikzpicture}
   [scale=1,auto=left,every node/.style={circle,scale=0.9}]
  \node[draw,circle,fill=black,label=below:$1$] (o) at (2,5) {};
  
  \node[draw,circle,fill=black, label=below:$1$] (n) at (1.5,5) {};
   \node[draw,circle,fill=black, label=below:$1$] (n1) at (1,5) {};

  \node[draw,circle,fill=black, label=below:$1$] (l) at (3,5) {};
 \node[draw,circle,fill=black, label=below:$1$] (v) at (2.5,5) {};
 \node[draw,circle,fill=black, label=below:$1$] (v1) at (3.5,5) {};
 \node[draw,circle,fill=black, label=below:$1$] (v2) at (4,5) {};
  \node[draw, circle, fill=blue!10, inner sep=0, label=right:$7$] (j) at (3,6) {$\cup$};
  \node[draw,circle,fill=blue!10, inner sep=0, label=left:$12$] (h) at (2,7) {$\otimes$};

  \node[draw, circle, fill=blue!10, inner sep=0, label=left:$5$] (g) at (1,6) {$\cup$};
    
  \node[draw,circle,fill=blue!10, inner sep=0, label=left:$9$] (a) at (0,5) {$\otimes$};
  
  \node[draw, circle, fill=blue!10, inner sep=0, label=right:$7$] (g1) at (0.25,4) {$\cup$};
  \node[draw, circle, fill=black, label=below:$1$] (b) at (0.25,3) {};
  \node[draw,circle,fill=black, label=below:$1$] (c) at (0.75,3) {};
  \node[draw,circle,fill=black, label=below:$1$] (e) at (0.5,5) {};
  
   \node[draw, circle, fill=blue!10, inner sep=0,label=left:$7$] (g2) at (-0.5,4) {$\cup$};
 \node[draw,circle,fill=blue,label=below:$1$] (e1) at (-1,3) {};
  \node[draw,circle,fill=blue,label=below:$1$] (e2) at (-0.25,3) {};

  \path 
        (a) edge node[below]{} (g2)
(e1) edge node[below]{} (g2)
(e2) edge node[below]{} (g2)

(a) edge node[left]{} (g1)
(g1) edge node[left]{} (b)
(g1) edge node[left]{} (c)

  (g) edge node[below]{} (e)
  (g) edge node[below]{} (n)
 (g) edge node[below]{} (n1)

    (j) edge node[below]{} (v)
    (j) edge node[below]{} (v1)
(j) edge node[below]{} (v2)      
        (g) edge node[left]{}(a)
        (h) edge node[right]{}(j)
        (h) edge node[left]{}(g)

        (j) edge node[right]{}(l)
       (j) edge node[below]{}(o);
\end{tikzpicture}
\caption{ First iteration}
       \label{fig5}
       \end{minipage}\hfill
       \begin{minipage}[c]{0.4 \linewidth}
\begin{tikzpicture}
   [scale=1,auto=left,every node/.style={circle,scale=0.9}]
  \node[draw,circle,fill=black,label=below:$1$] (o) at (2,5) {};
  
  \node[draw,circle,fill=black, label=below:$1$] (n) at (1.5,5) {};
   \node[draw,circle,fill=black, label=below:$1$] (n1) at (1,5) {};

  \node[draw,circle,fill=black, label=below:$1$] (l) at (3,5) {};
 \node[draw,circle,fill=black, label=below:$1$] (v) at (2.5,5) {};
 \node[draw,circle,fill=black, label=below:$1$] (v1) at (3.5,5) {};
 \node[draw,circle,fill=black, label=below:$1$] (v2) at (4,5) {};
  \node[draw, circle, fill=blue!10, inner sep=0,  label=right:$7$] (j) at (3,6) {$\cup$};
  \node[draw,circle,fill=blue!10, inner sep=0,  label=left:$12$] (h) at (2,7) {$\otimes$};

  \node[draw, circle, fill=blue!10, inner sep=0,  label=left:$5$] (g) at (1,6) {$\cup$};
    
  \node[draw,circle,fill=blue!10, inner sep=0, label=left:$9$] (a) at (0,5) {$\otimes$};
  
  \node[draw, circle, fill=blue!10, inner sep=0,label=right:$7$] (g1) at (0.25,4) {$\cup$};
  \node[draw, circle, fill=blue, label=below:$1$] (b) at (0.25,3) {};
  \node[draw,circle,fill=blue, label=below:$1$] (c) at (0.75,3) {};
  \node[draw,circle,fill=black, label=below:$1$] (e) at (0.5,5) {};
  
   \node[draw, circle, fill=black,label=left:$2$] (g2) at (-0.5,4) {};
 \node[draw,circle,fill=black,label=below:$7$] (e1) at (-1,3) {};

  \path 
        (a) edge node[below]{} (g2)

(a) edge node[left]{} (g1)
(g1) edge node[left]{} (b)
(g1) edge node[left]{} (c)

  (g) edge node[below]{} (e)
  (g) edge node[below]{} (n)
 (g) edge node[below]{} (n1)

    (j) edge node[below]{} (v)
    (j) edge node[below]{} (v1)
(j) edge node[below]{} (v2)      
        (g) edge node[left]{}(a)
        (h) edge node[right]{}(j)
        (h) edge node[left]{}(g)

        (j) edge node[right]{}(l)
       (j) edge node[below]{}(o);
\end{tikzpicture}
       \caption{Second iteration}
       \label{fig6}
        \end{minipage}
\end{figure}

In the next step, again duplicate vertices $\{v_1, v_2\}$ are chosen of $w_i$ with maximum depth.
Since $\delta(w_i)=7$  the assignments are given
$$ v_1 \leftarrow 7 ;\hspace{1cm}v_2 \leftarrow 1+1=2; \hspace{1cm} T_{G} = T_{G} - v_1.$$ 
Then the vertex $v_1$ is removed and $v_2$ is relocated under its parent's, represented in Figure \ref{fig7}.
It means that $7$ is a Laplacian eigenvalue with multiplicity one.

In the next step a pair of coduplicate vertices $\{v_1, v_2\}$ of depth three is chosen. Since $\delta(w_i)=9$ 
and the vertices have weiths $p_i=2,$ the following assigments are made 
$$ v_1 \leftarrow 9 ;\hspace{1cm}v_2 \leftarrow 2+2=4; \hspace{1cm} T_{G} = T_{G} - v_1.$$ 

\begin{figure}[h!]
       \begin{minipage}[c]{0.45 \linewidth}
\begin{tikzpicture}
   [scale=1,auto=left,every node/.style={circle,scale=0.9}]

   [scale=1,auto=left,every node/.style={circle,scale=0.9}]
  \node[draw,circle,fill=black,label=below:$1$] (o) at (2,5) {};
  
  \node[draw,circle,fill=black, label=below:$1$] (n) at (1.5,5) {};
   \node[draw,circle,fill=black, label=below:$1$] (n1) at (1,5) {};

  \node[draw,circle,fill=black, label=below:$1$] (l) at (3,5) {};
 \node[draw,circle,fill=black, label=below:$1$] (v) at (2.5,5) {};
 \node[draw,circle,fill=black, label=below:$1$] (v1) at (3.5,5) {};
 \node[draw,circle,fill=black, label=below:$1$] (v2) at (4,5) {};
  \node[draw, circle, fill=blue!10, inner sep=0,  label=right:$7$] (j) at (3,6) {$\cup$};
  \node[draw,circle,fill=blue!10, inner sep=0,  label=left:$12$] (h) at (2,7) {$\otimes$};

  \node[draw, circle, fill=blue!10, inner sep=0, label=left:$5$] (g) at (1,6) {$\cup$};
    
  \node[draw,circle,fill=blue!10, inner sep=0,label=left:$9$] (a) at (0,5) {$\otimes$};
  
  \node[draw, circle, fill=blue,label=right:$2$] (g1) at (0.25,4) {};
  \node[draw, circle, fill=black, label=below:$7$] (b) at (0.25,3) {};
  \node[draw,circle,fill=black, label=below:$1$] (e) at (0.5,5) {};
  
   \node[draw, circle, fill=blue,label=left:$2$] (g2) at (-0.5,4) {};
 \node[draw,circle,fill=black,label=below:$7$] (e1) at (-1,3) {};

  \path 
        (a) edge node[below]{} (g2)

(a) edge node[left]{} (g1)

  (g) edge node[below]{} (e)
  (g) edge node[below]{} (n)
 (g) edge node[below]{} (n1)

    (j) edge node[below]{} (v)
    (j) edge node[below]{} (v1)
(j) edge node[below]{} (v2)      
        (g) edge node[left]{}(a)
        (h) edge node[right]{}(j)
        (h) edge node[left]{}(g)

        (j) edge node[right]{}(l)
       (j) edge node[below]{}(o);\end{tikzpicture}
\caption{ Third iteration}
       \label{fig7}
       \end{minipage}\hfill
       \begin{minipage}[c]{0.4 \linewidth}
\begin{tikzpicture}
[scale=1,auto=left,every node/.style={circle,scale=0.9}]

   [scale=1,auto=left,every node/.style={circle,scale=0.9}]
  \node[draw,circle,fill=black,label=below:$1$] (o) at (2,5) {};
  
  \node[draw,circle,fill=blue, label=below:$1$] (n) at (1.5,5) {};
   \node[draw,circle,fill=blue, label=below:$1$] (n1) at (1,5) {};

  \node[draw,circle,fill=black, label=below:$1$] (l) at (3,5) {};
 \node[draw,circle,fill=black, label=below:$1$] (v) at (2.5,5) {};
 \node[draw,circle,fill=black, label=below:$1$] (v1) at (3.5,5) {};
 \node[draw,circle,fill=black, label=below:$1$] (v2) at (4,5) {};
  \node[draw, circle, fill=blue!10, inner sep=0,  label=right:$7$] (j) at (3,6) {$\cup$};
  \node[draw,circle,fill=blue!10, inner sep=0,  label=left:$12$] (h) at (2,7) {$\otimes$};

  \node[draw, circle, fill=blue!10, inner sep=0,label=left:$5$] (g) at (1,6) {$\cup$};
    
  \node[draw,circle,fill=blue,label=below:$4$] (a) at (0,5) {};
  
  \node[draw, circle, fill=black, label=below:$7$] (b) at (0,3) {};
  \node[draw,circle,fill=blue, label=below:$1$] (e) at (0.5,5) {};
  
   \node[draw, circle, fill=black,label=below:$9$] (g2) at (1,3) {};
 \node[draw,circle,fill=black,label=below:$7$] (e1) at (-1,3) {};

  \path 


  (g) edge node[below]{} (e)
  (g) edge node[below]{} (n)
 (g) edge node[below]{} (n1)

    (j) edge node[below]{} (v)
    (j) edge node[below]{} (v1)
(j) edge node[below]{} (v2)      
        (g) edge node[left]{}(a)
        (h) edge node[right]{}(j)
        (h) edge node[left]{}(g)

        (j) edge node[right]{}(l)
       (j) edge node[below]{}(o);\end{tikzpicture}
       \caption{Fourth iteration}
       \label{fig8}
        \end{minipage}
\end{figure}

 Figure \ref{fig8} shows the cotree with $v_1$ removed and $v_2$ relocated.
It means that $9$ is a Laplacian eigenvalue with multiplicity one. 
Next, a set of duplicate vertices $\{v_1, v_2, v_3, v_4\}$ with depth two is selected. 
Since $\delta(w_i)=5$ we have the following  assigments
$$ v_1,v_2, v_3 \leftarrow 5; \hspace{0,5cm} v_4 \leftarrow 4+1+1+1=7;  \hspace{0,5cm}T_{G} -(v_1,v_2,v_3).$$

\begin{figure}[h!]
       \begin{minipage}[c]{0.45 \linewidth}
\begin{tikzpicture}
 [scale=1,auto=left,every node/.style={circle,scale=0.9}]

   [scale=1,auto=left,every node/.style={circle,scale=0.9}]
  \node[draw,circle,fill=blue,label=below:$1$] (o) at (2,5) {};
  
  \node[draw,circle,fill=black, label=below:$5$] (n) at (1.25,3) {};
   \node[draw,circle,fill=black, label=below:$5$] (n1) at (2,3) {};

  \node[draw,circle,fill=blue, label=below:$1$] (l) at (3,5) {};
 \node[draw,circle,fill=blue, label=below:$1$] (v) at (2.5,5) {};
 \node[draw,circle,fill=blue, label=below:$1$] (v1) at (3.5,5) {};
 \node[draw,circle,fill=blue, label=below:$1$] (v2) at (4,5) {};
  \node[draw, circle, fill=blue!10, inner sep=0,label=right:$7$] (j) at (3,6) {$\cup$};
  \node[draw,circle,fill=blue!10, inner sep=0,  label=left:$12$] (h) at (2,7) {$\otimes$};

  \node[draw, circle, fill=black,label=left:$7$] (g) at (1,6) {};
    
  \node[draw,circle,fill=black,label=below:$5$] (a) at (2.75,3) {};
  
  \node[draw, circle, fill=black, label=below:$7$] (b) at (-0.25,3) {};
  
   \node[draw, circle, fill=black,label=below:$9$] (g2) at (0.5,3) {};
 \node[draw,circle,fill=black,label=below:$7$] (e1) at (-1,3) {};

  \path 



    (j) edge node[below]{} (v)
    (j) edge node[below]{} (v1)
(j) edge node[below]{} (v2)      
        (h) edge node[right]{}(j)
        (h) edge node[left]{}(g)

        (j) edge node[right]{}(l)
       (j) edge node[below]{}(o);\end{tikzpicture}
\caption{ Fiveth iteration}
       \label{fig9}
       \end{minipage}\hfill
       \begin{minipage}[c]{0.4 \linewidth}
\begin{tikzpicture}
 [scale=1,auto=left,every node/.style={circle,scale=0.9}]

   [scale=1,auto=left,every node/.style={circle,scale=0.9}]
  
  \node[draw,circle,fill=black, label=below:$5$] (n) at (1.25,3) {};
   \node[draw,circle,fill=black, label=below:$5$] (n1) at (2,3) {};

  \node[draw,circle,fill=black, label=below:$7$] (l) at (5.75,3) {};
 \node[draw,circle,fill=black, label=below:$7$] (v) at (5,3) {};
 \node[draw,circle,fill=black, label=below:$7$] (v1) at (3.5,3) {};
 \node[draw,circle,fill=black, label=below:$7$] (v2) at (4.25,3) {};
  \node[draw, circle, fill=blue,label=right:$5$] (j) at (3,6) {};
  \node[draw,circle,fill=blue!10, inner sep=0, label=right:$12$] (h) at (2,7) {$\otimes$};

  \node[draw, circle, fill=blue,label=left:$7$] (g) at (1,6) {};
    
  \node[draw,circle,fill=black,label=below:$5$] (a) at (2.75,3) {};
  
  \node[draw, circle, fill=black, label=below:$7$] (b) at (-0.25,3) {};
  
   \node[draw, circle, fill=black,label=below:$9$] (g2) at (0.5,3) {};
 \node[draw,circle,fill=black,label=below:$7$] (e1) at (-1,3) {};

  \path 



        (h) edge node[right]{}(j)
        (h) edge node[left]{}(g);


\end{tikzpicture}
       \caption{Sixth iteration}
       \label{fig10}
        \end{minipage}
\end{figure}

Figure \ref{fig9} shows the cotree with $\{v_1, v_2, v_3\}$ removed and $v_4$ relocated. It means that $5$ is a Laplacian eigenvalue with multiplicity 3.
Next, another depth  two of duplicate vertices $\{v_1, v_2, v_3, v_4, v_5\}$  is selected. 
Since $\delta(w_i)=7$ we have the following  assigments
$$ v_1,v_2, v_3, v_4 \leftarrow 7; \hspace{0,5cm} v_5 \leftarrow 1+1+1+1+1=5;  \hspace{0,5cm}T_{G} -(v_1,v_2,v_3,v_4).$$
Figure \ref{fig10} shows the cotree with $\{v_1, v_2, v_3, v_4\}$ removed and $v_5$ relocated. It means that $7$ is a Laplacian eigenvalue with multiplicity 4.

\begin{figure}[h!]
\begin{tikzpicture}
 [scale=1,auto=left,every node/.style={circle,scale=0.9}]

   [scale=1,auto=left,every node/.style={circle,scale=0.9}]
  
  \node[draw,circle,fill=black, label=below:$5$] (n) at (1.25,3) {};
   \node[draw,circle,fill=black, label=below:$5$] (n1) at (2,3) {};

  \node[draw,circle,fill=black, label=below:$7$] (l) at (5.75,3) {};
 \node[draw,circle,fill=black, label=below:$7$] (v) at (5,3) {};
 \node[draw,circle,fill=black, label=below:$7$] (v1) at (3.5,3) {};
 \node[draw,circle,fill=black, label=below:$7$] (v2) at (4.25,3) {};
  \node[draw, circle, fill=black,label=right:$0$] (j) at (3,6) {};
  \node[draw,circle,fill=blue!10, inner sep=0, label=right:$$] (h) at (2,7) {$\otimes$};

  \node[draw, circle, fill=black,label=left:$12$] (g) at (1,6) {};
    
  \node[draw,circle,fill=black,label=below:$5$] (a) at (2.75,3) {};
  
  \node[draw, circle, fill=black, label=below:$7$] (b) at (-0.25,3) {};
  
   \node[draw, circle, fill=black,label=below:$9$] (g2) at (0.5,3) {};
 \node[draw,circle,fill=black,label=below:$7$] (e1) at (-1,3) {};

  \path 



        (h) edge node[right]{}(j)
        (h) edge[dashed] node[left]{}(g);


\end{tikzpicture}
       \caption{ Final step }
       \label{fig11}
\end{figure}

Finally, the last step of Algorithm $L$-eigenvalues is to process the coduplicate vertices $\{v_1,v_2\}$ with assignments $p_1=7$ and $p_2=5,$ represented in  Figure \ref{fig10}.
Since the cotree has depth 1 and $\delta(w_i)=12,$  we have the following  assigments  are made
$$ v_1 \leftarrow 12; \hspace{0,5cm} v_2 \leftarrow 0.$$
It means that $12$ and $0$ are Laplacian eigenvalues with multiplicities one. This step is illustrated in  Figure \ref{fig11}.
The algorithm stops, since the cotree has only one leaf, and we have that  $Spect_{L}(G)= \{12, 9,7^6, 5^3, 0\}.$
\end{Ex}

\section{The number of spanning trees of  a cograph}
\label{Sec4}
 We recall that a spanning tree of   a connected  undirected graph $G$ on $n$ vertices is  a connected  $(n-1)$-edge subgraph of $G.$
 We finalize this paper using the Algorithm $L$-eigenvalues
as a tool  for obtaining a closed formula for the number of spanning trees of  a cograph.

The following result is due \cite{Kelmans}.

\begin{Lem}
\label{Lem5}
Let $G$ be a connected graph on $n$ vertices and let $L(G)$ be its Laplacian matrix. Let $0= \mu_1 \leq \mu_2 \leq \ldots \leq \mu_n$ be the eigenvalues of $L(G).$
Then the number of spanning trees of $G$ equals the product $\prod_{i=2}^{n} \mu_i /n .$  
 \end{Lem}


\begin{Thr}
\label{main3}
Let $G$ be a connected cograph on $n$ vertices with cotree $T_G$ having $r$ interior vertices $\{w_1, \ldots, w_r\}.$ 
Let $w_1$ be the cotree's root and  suppose that any interior vertex $w_i$ of $T_G$ has $s_i$ successors vertices  and $t_i$ leaves. Then the number of spanning trees of $G$ equals
$$  \prod_{i=2}^{r}(\delta(w_i))^{s_i + t_i -1} \cdot n^{s_1 +t_1 -2}$$
where $\delta(w_i)$ is the degree of interior vertex $w_i.$
\end{Thr}

\begin{proof}
Let $G$ be a connected cograph on $n$ vertices with cotree $T_G$ having  $r$ interior vertices $\{w_1, \ldots, w_r\}.$ We assume that $w_1$ is the cotree's root and each $w_i$ has $s_i$ successors vertices and $t_i$ leaves. It is sufficient to show that $\delta(w_i)$ correspondent to a Laplacian eigenvalue of $G$ with multiplicity $s_i +t_i-1.$ 

If $w_i =\cup$  and it is a terminal vertex, we have that $\delta(w_i)$ coincides with the degree of vertices in the leaves $t_i.$ So, by Lemma \ref{Lem1} we are done. Now, if $w_i$ is not a terminal vertex, we have that $w_i$ has $s_i$ successors vertices and $t_i$ leaves. After the Algorithm $L$-eigenvalues processes the $s_i$ interior vertices, by Lemma \ref{Lem3}, $\delta(w_i)$ correspondent a Laplacian eigenvalue of $G$ with multiplicity $s_i +t_i-1.$
The proof is analogous for an interior vertex $w_i$ of $\otimes$-type.

To finish, we need to account the Laplacian eigenvalue given by the interior vertex of cotree's root.
Let $w_1=\otimes$ be the interior vertex of cotree's root of $T_G$ having $s_1$ successors vertices and $t_1$ leaves. 
Since the Algorithm $L$-eigenvalues  assigns to $w_1$ the Laplacian eigenvalue $n$ with multiplicity $s_1 +t_1-1$ and according Lemma \ref{Lem5} we must to divide the product of Laplacian eigenvalues  by factor $n,$
follows $n$ is a factor with multiplicity $s_1 +t_1-2,$ as desired.
\end{proof}

As illustration, we exhibit an example for obtain the number of spanning trees of a cograph.

\begin{Ex}
Let $G$ be a cograph having cotree $T_G$ as the Figure \ref{fig12} has shown.

\begin{figure}[h!]
\begin{tikzpicture}
   [scale=1,auto=left,every node/.style={circle,scale=0.9}]
  \node[draw,circle,fill=black,label=below:$$] (o) at (2,5) {};
  
  \node[draw,circle,fill=black, label=below:$$] (n) at (1.5,5) {};
   \node[draw,circle,fill=black, label=below:$$] (n1) at (1,5) {};

  \node[draw,circle,fill=black, label=below:$$] (l) at (3,5) {};
 \node[draw,circle,fill=black, label=below:$$] (v) at (2.5,5) {};
 \node[draw,circle,fill=black, label=below:$$] (v1) at (3.5,5) {};
 \node[draw,circle,fill=black, label=below:$$] (v2) at (4,5) {};
  \node[draw, circle, fill=blue!10, inner sep=0, label=left:$w_3$] (j) at (3,6) {$\cup$};
  \node[draw,circle,fill=blue!10, inner sep=0, label=left:$w_1$] (h) at (2,7) {$\otimes$};

  \node[draw, circle, fill=blue!10, inner sep=0, label=left:$w_2$] (g) at (1,6) {$\cup$};
    
  \node[draw,circle,fill=blue!10, inner sep=0, label=left:$w_4$] (a) at (0,5) {$\otimes$};
  
  \node[draw, circle, fill=blue!10, inner sep=0, label=right:$w_6$] (g1) at (0.25,4) {$\cup$};
  \node[draw, circle, fill=black, label=below:$$] (b) at (0.25,3) {};
  \node[draw,circle,fill=black, label=below:$$] (c) at (0.75,3) {};
  \node[draw,circle,fill=black, label=below:$$] (e) at (0.5,5) {};
  
   \node[draw, circle, fill=blue!10, inner sep=0, label=left:$w_5$] (g2) at (-0.5,4) {$\cup$};
 \node[draw,circle,fill=black,label=below:$$] (e1) at (-1,3) {};
  \node[draw,circle,fill=black,label=below:$$] (e2) at (-0.25,3) {};

  \path 
        (a) edge node[below]{} (g2)
(e1) edge node[below]{} (g2)
(e2) edge node[below]{} (g2)

(a) edge node[left]{} (g1)
(g1) edge node[left]{} (b)
(g1) edge node[left]{} (c)

  (g) edge node[below]{} (e)
  (g) edge node[below]{} (n)
 (g) edge node[below]{} (n1)

    (j) edge node[below]{} (v)
    (j) edge node[below]{} (v1)
(j) edge node[below]{} (v2)      
        (g) edge node[left]{}(a)
        (h) edge node[right]{}(j)
        (h) edge node[left]{}(g)

        (j) edge node[right]{}(l)
       (j) edge node[below]{}(o);
\end{tikzpicture}
       \caption{ Cotree $T_G$ }
       \label{fig12}
\end{figure}

We have that $T_G$ has a set of interior vertices  $\{w_1, w_2, w_3, w_4, w_5, w_6\}$ where $w_1$ is the cotree's root of  $T_G.$
Computing the degrees of interior vertices and their respective number of successors vertices and leaves, we have
$$\delta(w_1)= 12, s_1=2, t_1=0; \delta(w_2)=5, s_2=1, t_2=3;\delta(w_3)=7,s_3=0, t_3=5;$$
$$\delta(w_4)=9, s_4=2, t_4=0;\delta(w_5)=7, s_5=0, t_5=2;  \delta(w_6)=7, s_6=0, t_6=2;$$
By Theorem \ref{main3}, we have the number of spanning trees of $G$ equals
$$ (5)^3 \cdot (7)^4 \cdot (9)^1 \cdot (7)^1\cdot (7)^1\cdot (12)^0 .$$
\end{Ex}

\section{Acknowledgements}
This work was part of the Post-Doctoral studies of Fernando C. Tura, while visiting Georgia State University, on leave from UFSM and supported by CNPq Grant
200716/2022-0. Guantao Chen acknowledges partial support of NSF grant DMS-2154331.




\begin{thebibliography}{99}




\bibitem{Abreu2}
Abreu, N.,  Justel, C.M., Markenzon, L.: Integer Laplacian eigenvalues of chordal graphs, Linear Algebra Appl. {\bf 614},  68--81 (2021)


\bibitem{Milica}
Andeli\'c, M., Du, Z., da Fonseca, C.M., Simi\'c, S.K.: Tridiagonal matrices and spectral properties of
some graph classes. J. Czech. Math. {\bf 70}, 1125--1138 (2020)

\bibitem{Allem2}
 Allem, L.E.,  Tura, F.C.:  Integral cographs,
Discrete Appl. Math.  {\bf 283},  153--167 (2020)


\bibitem{Allem}
Allem, L. E., Tura, F. C.:  Multiplicity of eigenvalues of cographs,
Discrete Appl. Math.  {\bf 247},  43--52 (2018)






\bibitem{Bapat} Bapat, R. B.,  Lal, A. K., Pati, S.: Laplacian spectrum of weakly quasi-threshold graphs, Graphs Combin. {\bf 24},  273--290 (2008)


\bibitem{BSS2011}
B{\i}y{\i}ko\u{g}lu, T., Simi\'{c}, S. K.,  Stani\'c, Z.: Some notes
on spectra of cographs, Ars Combinatoria  {\bf 100}, 421--434 (2011)

\bibitem{Paul} Bretscher, A., Corneil, D., Habib, M., Paul, C.: A simple linear time LexBFS cograph recognition
algorithm. SIAM J. Discrete Math. {\bf 22}, 1277--1296 (2008)

\bibitem{Le} Brandstadt, A., Le, V.B., Spinrad, J.P.: Graph classes: a survey. In: SIAM Monographs on Discrete
Mathematics and Applications (1999)


\bibitem{Stewart}
Corneil, D. G.,  Lerchs, H.,  Stewart Bhirmingham, L.:  Complement reducible
graphs, Discrete. Appl.  Math.  {\bf 3} , 163--174 (1981)

\bibitem{Stewart2}
Corneil, D. G.,  Perl, Y.,  Stewart Bhirmingham, L.:  A linear recognition algorithm for cographs, SIAM J. Comput.  {\bf 14(4)},  926--934 (1985)


\bibitem{Golu} Golumbic, M.C.: Algorithmic graph theory and perfect graphs, 2nd edn. In: Annals of Discrete Mathematics, vol. 57. Elsevier, Amsterdam (2004)

\bibitem{hamer}
Hammer, P.L., Kelmans, A. K.: Laplacian spectra and spanning trees of  threshold 
graphs, Discrete Appl. Math. {\bf 65},  255--273 (1996)

\bibitem{Michel}
 Habib, M., Paul, C.: A simple linear time algorithm for cograph recognition, Discrete Appl. Math. {\bf 145 (2)}, 183--197 (2005).


\bibitem{JTT2016}
Jacobs, D.P.,  Trevisan, V.,  Tura, F. C.: Eigenvalue location in cographs, Discrete Appl. Math. {\bf 245},  220--235  (2018)


\bibitem{Jones}
Jones, A.,  Trevisan,V.,   Vinagre, C. T. M.: Exploring symmetries in cographs: Obtaining spectra and energies, Discrete Appl. Math. {\bf 325},   120--133 (2023)


\bibitem{Kelmans}
 Kelmans, A.K., Chelnokov, V.M.: A certain polynomial of a graph and graphs with an
extremal number of trees. J. Combin. Theory (B) {\bf 16}, 197--214 (1974)


\bibitem{tura22}
Lazzarin, J.,  Sosa, O. F.,   Tura, F. C.: Laplacian eigenvalues of equivalent cographs,
Linear Multilinear Algebra {\bf 71:6}, 1003--1014 (2022)


\bibitem{JOT2019}
Lazzarin, J.,   M\'arquez, O. F.,  Tura, F. C.: No threshold graphs are cospectral,
Linear Algebra Appl. {\bf 560},  133--145. (2019)




\bibitem{Mah95} Mahadev, N.V.R.,   Peled, U.N.:
Threshold graphs and related topics, Elsevier, 1995.




\bibitem{Merris}
Merris, R.: Degree maximal graphs are Laplacian integral,
Linear Algebra  Appl. {\bf 199}, 381--389 (1994)





\bibitem{Mousavi}
Mousavi, S. S.,  Haeri, M.,   Mesbahi, M.:  Laplacian Dynamics on Cographs: Controllability Analysis
Through Joins and Unions, IEEE Transactions on Automatic Control, {\bf 66(3)}, 1383--1390 (2021)



\bibitem{nikolo3}
Stavros D. Nikolopoulos, Charis Papadopoulos.:
A simple linear-time recognition algorithm for weakly quasi-threshold graphs.
Graphs  Combin. {\bf 27}, 557--565 (2011)

\bibitem{Niko}
Stavros D. Nikolopoulos, Charis Papadopoulos.:
Counting spanning trees in cographs: an algorithmic
approach, Ars  Combinatoria,  (2009)









\bibitem{Yan}Jing-Ho,  Y.,  Jer-Jeong, C.,  Gerard, J.C.:  Quasi-threshold graphs, Discrete Appl. Math. {\bf 69}, 247--255 (1996)










\end{thebibliography}
\end{document}